\documentclass{amsart}
\usepackage{xypic}
\usepackage[all,cmtip]{xy}

\usepackage{amsmath, amssymb, amsthm, amsfonts, mathrsfs, amsfonts}
\usepackage{cases}

\newcommand{\Z}{\ensuremath{\mathbb{Z}}}

\newtheorem{thm}{Theorem}[section]
\newtheorem*{thm2}{Theorem}

\newtheorem{lemma}[thm]{Lemma}
\newtheorem{prop}[thm]{Proposition}
\theoremstyle{remark}
\newtheorem{rmk}{Remark}[section]
\newtheorem*{rmk2}{Remark}

\theoremstyle{definition}
\newtheorem{defin}{Definition}[section]
\newtheorem*{defin2}{Definition}

\newcommand{\on}{\operatorname}

\begin{document}

\title{Tensor products of higher almost split sequences}
\author{Andrea Pasquali}
\address{Dept. of Mathematics, Uppsala University, P.O. Box 480, 751 06 Uppsala, Sweden}
\email{andrea.pasquali@math.uu.se}
\maketitle

\begin{abstract}
  We investigate how the higher almost split sequences over a tensor product of algebras are related to those over each factor.
  Herschend and Iyama give in \cite{HI11} a criterion for when the tensor product of an $n$-representation finite algebra and an $m$-representation finite algebra is $(n+m)$-representation finite.
In this case we give a complete description of the higher almost split sequences over the tensor product by expressing every higher almost split sequence 
as the mapping cone of a suitable chain map and using a natural notion of tensor product for chain maps.
\end{abstract}

\section{Introduction and conventions} 
In the context of Auslander-Reiten theory one can study almost split sequences of modules over a finite-dimensional algebra $A$. These are certain short exact sequences
$$0\to M\to N\to L\to 0$$ such that $M$ and $L$ are indecomposable, and it turns out that every nonprojective indecomposable module over $A$ appears as the last term of such a sequence
(and every noninjective indecomposable appears as the first term).
Moreover, such sequences are determined up to isomorphism by either the first or the last term (see for reference \cite{ASS06}).
One can do a similar construction in the context of higher dimensional Auslander-Reiten theory, at the cost of restricting to a suitable subcategory $\mathcal  C$ of $\on{mod}A$ that 
contains all injectives and all projectives. Then one gets longer so called $n$-almost split sequences
$$
0\to M\to X_1\to \cdots \to X_n\to L\to 0
$$
in $\mathcal  C$, and again every nonprojective module in $\mathcal  C$ appears at the end of such a sequence and every noninjective at the start of one.
Again, these sequences are determined by their first or last term (see \cite{Iya08}, \cite{Iya11}).
One of the most basic cases where such a situation appears is when $A$ is $n$-representation finite (cf.~\cite{HI11}, \cite{Iya08}).
\begin{defin2}
  \label{def:cluster}
	Let $A$ be a finite-dimensional $k$-algebra, and let $n \in \Z_{>0}$. 
	An \emph{$n$-cluster tilting module} for $A$ is a module $M_A\in \on{mod}A$ such that
	\begin{align*}
	  \on{add}M_A &= \left\{ X \in \on{mod}A \ | \ \on{Ext}^i_A(M_A,X) = 0 \text{ for every } 0<i<n \right\} = \\
			  &= \left\{ X \in \on{mod}A \ | \ \on{Ext}^i_A(X,M_A) = 0 \text{ for every } 0<i<n  \right\}.
	\end{align*}
	We say that $A$ is
	\emph{$n$-representation finite} if $\on{gl.dim}A \leq n$ and there exists an $n$-cluster tilting module for $A$.
	Then $\on{gl.dim}A = 0$ or $\on{gl.dim}A = n$.
\end{defin2}

For such algebras it is known that $\on{add}M_A$ is a subcategory of $\on{mod}A$ that admits $n$-almost split sequences.
We call $D$ the functor $D = \on{Hom}_k(-, k):\on{mod}A \to A\on{mod}$. 
The (higher) \emph{Auslander-Reiten translations} $\tau_n, \tau_n^-$ are defined as follows:
\begin{align*}
	\tau_n &= D\on{Ext}^n_A(-, A): \on{mod}A \to \on{mod}A\\
	\tau_n^-&= \on{Ext}^n_A(DA, -): \on{mod}A \to \on{mod}A.
\end{align*}
It is immediate from this definition that
\begin{align*}
  \tau_n A = 0 = \tau_n^- DA.
\end{align*}
These higher Auslander-Reiten translations behave similarly to the classical ones.
\begin{thm2}
	Let $A$ be an $n$-representation finite $k$-algebra. Let $P_1, \ldots, P_a$ be nonisomorphic representatives of 
	the isomorphism classes of indecomposable projective right $A$-modules, and $I_1, \ldots, I_a$ the corresponding indecomposable injective modules. Then:
	\begin{enumerate}
		\item There exist positive integers $l_1, \ldots, l_a$ and a permutation $\sigma\in S_a$ (the symmetric group over $a$ elements) such that $P_i \cong\tau_n^{l_i-1}I_{\sigma(i)}$
			for every $i$.
		\item There exists a unique (up to isomorphism) basic $n$-cluster tilting module $M_A$, which is given by
			$$
			M_A= \bigoplus_{i = 1}^a \bigoplus_{j=0}^{l_i-1} \tau_n^j I_{\sigma(i)}.
			$$
		\item The Auslander-Reiten translations induce mutually quasi-inverse equivalences
			$$
			\xymatrix{
				\on{add}(M_A/P) \ar[r]_{\tau_n}& \on{add}(M_A/I) \ar@<-1ex>[l]_{\tau_n^-}
			}
			$$
			where $P = \bigoplus_{i = 1}^a P_i$ and $I = \bigoplus_{i = 1}^aI_i$.
	\end{enumerate}
\end{thm2}

\begin{proof}
  See \cite[1.3(b)]{Iya11}.
\end{proof}

From the last point it follows in particular that the $n$-cluster tilting module can be equally described by
$$
M_A= \bigoplus_{i = 1}^a \bigoplus_{j=0}^{l_i-1} \tau_n^{-j} P_i.
$$
\begin{defin2}[\cite{HI11}]
	An $n$-representation finite algebra $A$ is said to be \emph{$l$-homogeneous} if with the above notation we have 
	$l_1=\cdots=l_a=l$. 
\end{defin2}

If $A$ is $n$-representation finite, the category $\on{add}M_A$ decomposes into ``slices'', in the sense that 
every $X\in \on{add}M_A$ can be written uniquely as $X \cong \bigoplus_{i \geq 0 } X_i$, where each $X_i\in \on{add}\tau_n^{-i}A$. 
If $A$ is $l$-homogeneous, then every slice $\on{add}\tau_n^{-j}A$, where  $0\leq j\leq l-1$, has the same number of isomorphism classes of indecomposables.

We denote by $\mathcal  D^b(\on{mod}A)$ the bounded derived category of $\on{mod}A$, and denote by $\varepsilon: \on{mod}A \to \mathcal  D^b(\on{mod}A)$ the natural inclusion. The Nakayama functors
\begin{align*}
  \nu&= -\overset{L}{\otimes}_ADA \cong D\circ R\on{Hom}_A(-, A): \mathcal  D^b(\on{mod}A)\to \mathcal  D^b(\on{mod}A)\\
  \nu^{-1} &=R\on{Hom}_{A^{op}}(D-,A) \cong R\on{Hom}_A(DA, -):\mathcal  D^b(\on{mod}A)\to \mathcal  D^b(\on{mod}A)
\end{align*}
are quasi-inverse equivalences that make the diagram
$$
\xymatrix{
  \mathcal  D^b(\on{mod}A) \ar[r]^\nu & \mathcal  D^b(\on{mod}A)\\ \mathcal K^b(\on{proj}A) \ar[u] \ar[r]^\nu &\mathcal K^b(\on{inj}A) \ar[u]
}
$$
commute ($\mathcal K^b$ denotes the bounded homotopy category).
If $A$ is $n$-representation finite, there is a natural isomorphism of functors $\on{mod}A \to \on{mod}A$
$$
\tau_n \cong H_0\circ\nu_n \circ\varepsilon $$
where $\nu_n = \nu\circ [-n]$.
For every $i$ and for every $0\leq j\leq l_i$, we have that $\varepsilon\tau_n^{-j}P_i = \nu_n^{-j}\varepsilon P_i$.
From now on, explicit mentions of $\varepsilon$ will be omitted for simplicity.

The definition of higher almost split sequences that is convenient to take is the following:
\begin{defin2}
	Let $A$ be an $n$-representation finite $k$-algebra, and let $M_A$ be the corresponding basic $n$-cluster tilting module.
	Let $$
	\xymatrix{
	  0 \ar[r]& C_{n+1} \ar[r]^{f_{n+1}}& C_{n} \ar[r] & \cdots \ar[r]& C_1 \ar[r]^{f_1}& C_0 \ar[r] &0
	}
	$$
	be an exact sequence with terms in $\on{add}M_A$. Such a sequence is an \emph{$n$-almost split sequence} if the
	following holds:
	\begin{enumerate}
		\item For every $i$, we have $f_i \in \on{rad}(C_i, C_{i-1})$.
		\item The modules $C_{n+1}$ and $C_0$ are indecomposable.
		\item \label{eq:cond} The sequence of functors  from $\on{add}M_A$ to $k\on{mod}$
			$$
			\xymatrix{
			  0 \ar[r]& \on{Hom}_A(-,C_{n+1}) \ar[r]^-{f_{n+1}\circ-}& 
				\on{Hom}_A(-,C_{n}) \ar[r] & \cdots \\
				&\cdots\ar[r] & \on{Hom}_A(-,C_1) \ar[r]^{f_1\circ-}
				& \on{rad}_A(-,C_0) \ar[r] &0
			}
			$$
			is exact (i.e.~it is an exact sequence when evaluated at any $X\in\on{add}M_A$).
	\end{enumerate}
\end{defin2}
\begin{thm2}
  \label{thm:exist}
	Let $A$ be an $n$-representation finite $k$-algebra, and let $M_A$ be the corresponding basic $n$-cluster tilting module. 
	Then we have the following:
	\begin{enumerate}
	  \item For every indecomposable nonprojective module $N \in \on{add}M_A$ there exists an $n$-almost split 
			sequence
			$$
			\xymatrix{
				0 \ar[r] & \tau_n N \ar[r] & \cdots \ar[r]& N\ar[r] &0,
			}
			$$
			and any $n$-almost split sequence whose last term is $N$ is isomorphic to this one.
		      \item For every indecomposable noninjective module $M \in \on{add}M_A$ there exists an $n$-almost split sequence
			$$
			\xymatrix{
				0\ar[r]&M \ar[r] &\cdots\ar[r]&\tau_n^-M\ar[r]&0,
			}
			$$
			and any $n$-almost split sequence whose first term is $M$ is isomorphic to this one.
	\end{enumerate}
\end{thm2}
\begin{proof}
  See \cite[Theorem 3.3.1]{Iya07}. Notice that the term ``$n$-cluster tilting subcategory'' has replaced ``$(n-1)$-orthogonal subcategory'' in recent literature. 
\end{proof}

\begin{rmk2}
  The usual, more general definition of $n$-almost split sequences that one takes requires that the condition dual to (\ref{eq:cond}) holds as well (as in \cite[Definition 2.1]{Iya11}).
However, in the case we are considering (module categories over an $n$-representation finite algebra), the two definitions are equivalent (see \cite[Proposition 2.10]{Iya08}).
\end{rmk2}

In their paper \cite{HI11}, Herschend and Iyama construct a class of examples of $n$-representation finite algebras via tensor products, in the setting where the ground field $k$ is perfect. 
Namely, they find a necessary and sufficient condition (being $l$-homogeneous for the same value of $l$) under which the tensor product $A\otimes B = 
A\otimes_k B$ of an $n$-representation finite algebra $A$ with an $m$-representation finite algebra $B$
is $(n+m)$-representation finite. They also show that in this case every indecomposable of $\on{add}M_{A\otimes B}$ is of the form 
$L\otimes N$ for some indecomposables $L\in \on{add}M_A$ and $N\in \on{add}M_B$, and that $\tau_{n+m}^{\pm}L\otimes N\cong \tau_n^{\pm}L\otimes \tau_m^{\pm}N$.
Moreover, in this case the algebra $A\otimes B$ is itself $l$-homogeneous.
\begin{rmk2}
  Even though not explicitly stated in \cite{HI11}, necessity of the condition comes from the following observation.
  Let  $$M= \bigoplus_{i, j} \bigoplus_d \tau_{n+m}^{-d} P_i\otimes Q_j$$ 
  where $P_i$ and $Q_j$ run over the indecomposable summands of $A, B$ respectively. 
  If $A$ and $B$ are not $l$-homogeneous for the same value of $l$, then $M$ has either an indecomposable summand of the form $S= L\otimes J$ where $J$ is injective and 
  $L$ is not, or one of the form $S=I\otimes N$ where $I$ is injective and $N$ is not. 
  On the other hand, if $A\otimes B$ is $(n+m)$-representation finite, then $M$ is an $(n+m)$-cluster tilting module, and hence the
  indecomposable injective $A\otimes B$-modules are precisely those indecomposable direct summands $I\otimes J$ of $M$ 
  such that $\tau_{n+m}^-I\otimes J = 0$. 
  Thus we reach a contradiction, since $S$ is not injective, but $\tau_{n+m}^-S = 0$.
\end{rmk2}
In this setting, if $$
0\to L\otimes N\to \cdots\to \tau_{n+m}^-L\otimes N \to 0$$
is an $(n+m)$-almost split sequence, then $\tau_{n+m}^-L\otimes N \cong \tau_n^-L\otimes \tau_m^-N$.
On the other hand, there are $n$- respectively $m$-almost split sequences
$$
0\to L\to \cdots\to \tau_n^-L\to 0$$
and 
$$
0\to N\to \dots\to \tau_m^-N\to 0,$$
so the starting and ending points behave well with respect to tensor products.
It is then a natural question to describe the relation between the sequence starting in $L\otimes N$ and the sequences starting in $L$ and $N$.
This is the question that we address, and we answer it in the setting where $A$ is $n$-representation finite, $l$-homogeneous and $B$ is $m$-representation finite, $l$-homogeneous.

For a precise statement, we need some more notation. 
 For a preadditive category $\mathcal  A$,
we denote by $\mathcal  C(\mathcal  A)$ the category of chain complexes of $\mathcal  A$.
If $A$ is a $k$-algebra and $\mathcal  A$ is a full subcategory of $\on{mod}A$, we denote by $\mathcal  C_r(\mathcal  A)$ 
the full subcategory of $\mathcal  C(\mathcal  A)$ whose objects are chain complexes where the differentials are radical morphisms (i.e.~$d_i\in
\on{rad}(A_i, A_{i-1})$ for every $i$).
Let $\mathcal  B$ be a full subcategory of $\mathcal  C(\mathcal  A)$.
We denote by $\on{Mor}(\mathcal  B)$ the category whose objects are chain maps $A_\bullet \to B_\bullet$ for $A_\bullet, B_\bullet\in 
\mathcal  B$, and whose morphisms are the obvious commutative diagrams. We denote by $\on{Mor}_r(\mathcal  B)$ the full
subcategory of $\on{Mor}(\mathcal  B)$ whose objects are radical chain maps $A_\bullet\to B_\bullet$ for $A
_\bullet,B_\bullet\in \mathcal  B$ (meaning that
for every $i$ the map $A_i \to B_i$ is radical).
We often view finite (exact) sequences as bounded chain complexes, and unless otherwise specified the degree-0 term is the 
rightmost nonzero term.
With this point of view in mind, we denote by $\mathcal  B^n$ the full subcategory of $\mathcal  B$ whose objects are complexes $C_\bullet$
satisfying $C_i = 0$ for every $i <0$ and $i >n$.
\begin{defin2}
  Let $A$ be an $n$-representation finite $k$-algebra, and let $i\in \Z_{\geq 0}$.
  Let $L\in\on{add}M_A$ be indecomposable noninjective, and let $C_\bullet$ be the corresponding 
	$n$-almost split sequence. Then we say that $C_\bullet$ \emph{starts in slice $i$} if $L \in \on{add}\tau_n^{-i}A$.
\end{defin2}
We denote by $\on{Cone} $ the mapping cone (see Definition \ref{def:cone}).
We use the symbol $\otimes^T$ for the usual ``total tensor product'' bifunctor
\begin{align*}
  -\otimes^T- : \mathcal C(\on{mod}A)\times \mathcal C(\on{mod}B)\to \mathcal C(\on{mod}A\otimes B)
\end{align*}
induced by $\otimes$ (see Section 3 for details). 
Our main result is the following:

\begin{thm}
  \label{thm:main}
  Let $k$ be a perfect field. 
  Let $A$ and $B$ be $n$- respectively $m$-representation finite $k$-algebras. Suppose that $A$ and $B$ are $l$-homogeneous for some common $l$.
  Let $\varphi\in \on{Mor}_r(\mathcal  C_r(\on{add}M_A))$ and let 
  $\psi \in \on{Mor}_r(\mathcal  C_r(\on{add}M_B))$. Suppose that $\on{Cone}(\varphi)$ and $\on{Cone}(\psi)$ are  $n$- respectively $m$-almost split sequences starting in slice $i$ 
  for some common $i \geq 0$. Then $\on{Cone}(\varphi\otimes^T\psi)$ is an $(n+m)$-almost split sequence.
\end{thm}

\begin{rmk2}
  In Theorem \ref{thm:every} we show that every $n$-almost split sequence is isomorphic to $\on{Cone}(\varphi)$ for some suitable $\varphi$, so all the $(n+m)$-almost split sequences 
  in $\on{mod}(A\otimes B)$ are obtained by this procedure.
\end{rmk2}

\begin{rmk2}
  The sequence $\on{Cone}(\varphi\otimes^T \psi)$ starts in slice $i$. This is because we have (see \cite{HI11})
  \begin{align*}
    \tau_n^{-i}A\otimes \tau_{m}^{-i}B = \tau_{n+m}^{-i}A\otimes B.
  \end{align*}
  On the other hand, if $L\in \on{add}\tau_n^{-i} A$ and $N\in\on{add}\tau_m^{-j}B$ with $i\neq j$, then $L\otimes N\not \in \on{add}M_{A\otimes B}$,
  so there is in principle no $(n+m)$-almost split sequence starting in $L\otimes N$.
\end{rmk2}

\begin{rmk2}
If we drop the condition guaranteeing that $A\otimes B$ is $(n+m)$-representation finite, then we can perform the same construction, and we still get some 
sequences in $\on{mod}(A\otimes B)$ which retain some interesting properties. Similarly, one could tensor sequences that do not start in the same slice.
This is a possible topic for future investigation.
\end{rmk2}

In Section 2 we show that every $n$-almost split sequence over an $n$-representation finite algebra is isomorphic to the mapping cone of a 
suitable chain map of complexes, then we relate the property of being $n$-almost split to a property of the chain map.
In Section 3 we define the functor $\otimes^T$ that we have mentioned above, and 
we prove the main theorem.
In Section 4 we compute an example where we explicitly construct a 2-almost split sequence and a 3-almost split sequence starting from a
1-representation finite algebra.

\paragraph{\textbf{Conventions.}}
Throughout this paper, we denote by $k$ a perfect field (cf.~\cite{HI11}). All $k$-algebras are associative and unitary.
For a ring $R$, we denote by $\on{mod}R$ (resp.~$R\on{mod}$) the category of finitely generated right (resp.~left)
$R$-modules. Unless otherwise specified, modules are right modules. Subcategory means full subcategory.
For a $k$-algebra $A$ we denote by $\on{rad}_A(-, -)$ the subfunctor of $\on{Hom}_A(-,-)$ defined by
\begin{align*}
  \on{rad}_{A}(X, Y) = \left\{ f\in\on{Hom}_{A}(X, Y)\ |\ \on{id}_X-g\circ f \text{ is invertible } \forall g\in \on{Hom}_A(Y, X) \right\}
\end{align*}
for all $A$-modules $X, Y$ (see \cite[Appendix 3]{ASS06}).
Thus $\on{rad}_A(-,-)$ is biadditive, and for two indecomposable modules $X\not\cong Y$ we have $\on{rad}_A(X, Y) = \on{Hom}_A(X, Y)$. 
Moreover, for an indecomposable module $X$ we have that $\on{rad}_A(X) := \on{rad}_A(X, X)$ is the Jacobson radical of the algebra $\on{End}_A(X)$.
We denote by $S_A(X, Y)$ the quotient $S_A(X, Y) = \on{Hom}_A(X, Y) /\on{rad}_A(X, Y)$ (and sometimes write only $S_A(X)$ instead of $S_A(X, X)$).
To simplify the notation, we sometimes omit the reference to the algebra when this is clear from the context (writing for instance $\on{Hom}$ instead of $\on{Hom}_A$).
For the rest of this paper, fix finite-dimensional $k$-algebras $A$ and $B$, where $A$ is $n$-representation finite and $B$ is $m$-representation finite.
Set $\mathcal A = \on{add}M_A$, $\mathcal  B = \on{add}M_B$, $\mathcal  A_i = \on{add}\tau_n^{-i}A$ for $i \geq 0$, and $\mathcal  B_j = \on{add}\tau_m^{-j}B$ for $j\geq 0$.

\section{$n$-almost split sequences as mapping cones}
\subsection{Preliminaries}

If $A$ is $n$-representation finite, then the morphisms in $\mathcal  A$ are ``directed'' with respect to the action of $\tau_n^-$.
More precisely, we have the following:

\begin{prop}
	\label{prop:hom}
	Let $A$ be an $n$-representation finite $k$-algebra. Let $M\in\mathcal A_i$ and $N\in \mathcal A_j$ with
	$i>j$. Then 
	$$
	\on{Hom}_A(M, N) = 0.$$
\end{prop}
\begin{proof}
	It is enough to check the result for $M, N$ indecomposable, i.e.~$M \cong \tau_n^{-i}P_1$ and $N\cong \tau_n^{-j}P_2$ for 
	some indecomposable projectives $P_1, P_2 \in \on{add}A$.
	We have
	\begin{align*}
	  \on{Hom}_A(M,N) &= \on{Hom}_{\mathcal  D^b(\on{mod}A)}(M,N) = \on{Hom}_{\mathcal  D^b(\on{mod}A)}
		(\nu_n^{-i}P_1, \nu_n^{-j}P_2) =\\&= \on{Hom}_{\mathcal  D^b(\on{mod}A)}(P_1, \nu_n^{i-j}P_2).
	\end{align*}
	In particular, $\on{Hom}_A(M,N)$ is a direct summand of (with the previous notation)
	\begin{align*}
		\bigoplus_{i=1}^a \on{Hom}_{\mathcal  D^b(\on{mod}A)}(P_i,\nu_n^{i-j}P_2 )  &= \on{Hom}_{\mathcal  D^b(\on{mod}A)}
		(A,\nu_n^{i-j}P_2 ) =\\&= H_0(\nu_n^{i-j}P_2) = \tau_n^{i-j}P_2 = 0
	\end{align*}
	since $i>j$ and $P_2$ is projective, so we are done.
\end{proof}

\begin{rmk}
  For $n = 1$, this is a special case of \cite[Corollary VIII.1.4]{ARS97}, since ``$1$-representation finite'' means ``hereditary and representation finite''.
\end{rmk}

We will be interested in checking whether a given complex is an $n$-almost split sequence, and for this purpose
it is convenient to take a slightly different point of view on the definition of $n$-almost splitness. Namely, fix an object $X \in \mathcal A$.
 We can define a functor
 $F_X: \mathcal  C_r(\mathcal A)\to \mathcal  C(k\on{mod})$ by mapping
$$
C_\bullet=
\xymatrix{
	\cdots \ar[r]^{f_{i+1}} & C_i \ar[r]^{f_i} & \cdots \ar[r]^{f_{1}} &  C_0 \ar[r]^{f_0} &\cdots
}
$$
to 
$$
F_X(C_\bullet) =
\xymatrix{
	\cdots \ar[r]^-{f_{i+1}\circ-} &
	\on{Hom}(X, C_i) \ar[r]^-{f_i \circ-} & \cdots \ar[r]^-{f_{1}\circ-} 
	& \on{rad}(X, C_0) \ar[r]^-{f_0\circ-}& \cdots
}
$$
(that is, $F_X$ is the subfunctor of $\on{Hom}(X, -)$ given by replacing $\on{Hom}(X, C_0)$ with $\on{rad}(X,C_0)$). 
This is well defined since $f_{1}$ is a radical morphism, hence the image of $f_1\circ-$ lies in $\on{rad}(X, C_0)$.
Then for a complex $C_\bullet \in \mathcal  C_r(\mathcal A)$ such that $C_i = 0 $ for $i>n+1$ and $i <0$, saying that it is 
an $n$-almost split sequence is equivalent to saying that $C_{n+1}$ and $C_0$ are indecomposable, $C_\bullet$ is exact, and $F_X(C_\bullet)$ is exact for every $X \in\mathcal A$ (or equivalently, for every indecomposable $
X\in \mathcal A$).
Similarly, we can define a subfunctor $G_X$ of the contravariant functor $\on{Hom}(-,X):\mathcal  C_r(\mathcal A) \to \mathcal  C(k\on{mod}) $ by mapping $C_\bullet$ to
$$
G_X(C_\bullet) = 
\xymatrix{
	\cdots \ar[r]^-{-\circ f_0} &
	\on{Hom}(C_0, X) \ar[r]^-{-\circ f_1} & \cdots \ar[r]^-{-\circ f_{n+1}} 
	& \on{rad}(C_{n+1},X) \ar[r]^-{-\circ f_{n+2}}& \cdots
}
$$
This is again well defined, and if $C_\bullet \in \mathcal  C_r(\mathcal A)$ is $n$-almost split then $G_X(C_\bullet)$ is 
exact for every $X \in \mathcal A$ (cf.~\cite[Proposition 2.10]{Iya08}).

\subsection{From sequences to cones}
\begin{defin}
  \label{def:cone}
	Let $\mathcal  D$ be an abelian category. 
	Let $A_\bullet \in \mathcal  C(\mathcal  D)$ with differentials $d_i: A_i \to A_{i-1}$. For any $m\in \Z$, the \emph{shift} $A[m]_\bullet$ of $A_\bullet$ is the complex 
	with objects $A[m]_i = A_{i+m}$ and differentials $d[m]_i: A[m]_i \to A[m]_{i-1}$ given by $d[m]_i = (-1)^m d_{i+m}$ for every $i$.
	
	Let $(A_\bullet, d_\bullet^A)$ and $(B_\bullet, d_\bullet^B)$ be complexes in $\mathcal  C(\mathcal  A)$. Let $f: A_\bullet \to B_\bullet$ 
	be a morphism of complexes with components $f_i:A_i \to B_i$. The \emph{shift} of $f$ is the morphism $f[m]: A_\bullet[m]\to B_\bullet[m]$
	with components $f[m]_i = f_{i+m}$. Thus $[m]$ is an endofunctor on $\mathcal C(\mathcal D)$.
	The \emph{mapping cone} $\on{Cone}(f)$ of $f$ is the complex with objects
	$$\on{Cone}(f)_i = A[-1]_i \oplus B_i$$
	and differentials
	$$
	d_i^{\on{Cone}(f)} = 
	\begin{bmatrix}
		d[-1]_i^A & 0\\
		f[-1]_i & d^B_i
	\end{bmatrix}.
	$$
\end{defin}

\begin{lemma}
	Let $\mathcal  D$ be an abelian category, and let 
	$f$ be a morphism of complexes in $\mathcal  C(\mathcal  D)$. Then $\on{Cone}(f)$ is exact if and only if $f$ is a quasi-isomorphism.
	\label{lem:qiso}
\end{lemma}

\begin{proof}
  This follows straight from \cite[III.18]{GM03}.
\end{proof}

Let $A$ be $n$-representation finite, and let
$$
	C_\bullet =
	\xymatrix{
		0 \ar[r]& C_{n+1} \ar[r]^{f_{n+1}}& C_{n} \ar[r] & \cdots \ar[r]& C_1 \ar[r]^{f_1}& C_0 \ar[r] &0
	}
	$$
	be an $n$-almost split sequence starting in slice $i_0$ for some $i_0\in \Z_{\geq 0}$. Then we can decompose the modules appearing in the sequence according to the slice
	decomposition of $\mathcal A$, i.e.~we write
	$$C_m = \bigoplus_{j\geq 0} B_m^j$$ with $B_m^j\in \mathcal A_j$ for every $m,j$. We know that $C_{n+1}\in\mathcal A_{i_0}$ and
	$C_{0}\in\mathcal A_{i_0+1}$ are indecomposable. A first result, which can be seen as a generalisation of \cite[Lemma VIII.1.8(b)]{ARS97}, is the following:

\begin{lemma}
  \label{lem:claim}
  With the above notation, we have
	\begin{align*} 
		B_m^{j} = 0
		\text{ for any } m \text{ and for } j \not\in \left\{ i_0, i_0+1 \right\}.
	\end{align*}
\end{lemma}

\begin{proof}
  To reach a contradiction, suppose that the claim is false. Then there is $B_q^j\neq0$ with $j \not\in\left\{ i_0, i_0+1 \right\}$.
	Suppose $j >i_0+1$, and pick $j$ maximal such.
	We can assume $q$ minimal for that value of $j$, i.e.~$B_{q-p}^j = 0$ for all $p>0$. Notice that since $C_0 = B_0^{i_0+1}$
	it follows that $q >0$. We want to prove that $C_\bullet$ cannot be $n$-almost split in this case, and it is enough to 
	show that $F_{B_q^j}(C_\bullet)$ is not exact.
	By Proposition \ref{prop:hom}, $$\on{Hom}(B_{p'}^j, B_p^i) = 0$$ for every $p, p'$ and for every $i <j$. By maximality of $j$, we get that $B^j_{\bullet}$ is a subcomplex of $C_\bullet$, and
	$$
	F_{B_q^j}(C_\bullet) = F_{B_q^j}(B_\bullet^j).
	$$
 	Since $q$ is minimal and $q>0$ we can write explicitly
	$$
	F_{B_q^j}(C_\bullet) = 
	\xymatrix{
	  \cdots \ar[r] &\on{Hom}(B_q^j, B_m^j) \ar[r] &\cdots \ar^-d[r] & \on{Hom}(B_q^j, B_q^j) \ar[r] & 0.
	}
	$$
	The map $d$ in this sequence is composition with a radical morphism, so in particular it cannot be surjective on $\on{Hom}(B_q^j, B_q^j)$. 
	The sequence is then not exact and we have proved that $B_m^j = 0$ for $j > i_0+1$. 
	
	Suppose now that $j <i_0$, and pick $j$ minimal such.
	We can assume that $q$ is maximal for that $j$, i.e.~$B_{q+p}^j = 0$ for all $p>0$. Notice that since $C_{n+1} = B_{n+1}^{i_0}$ it follows that $q<n+1$.
	We prove that $C_\bullet$ is not $n$-almost split in this case by showing that $G_{B_q^j}(C_\bullet)$ is not exact.
	Again by Proposition \ref{prop:hom} we know that $$\on{Hom}(B_p^i, B_{p'}^j) = 0$$ for all $p, p'$ if $i>j$. Then by minimality of $j$ and maximality of $q$ we get
	$$
	G_{B_q^j}(C_\bullet) = 
	\xymatrix{
	  \cdots \ar[r] & \on{Hom}(B_m^j, B_q^j) \ar[r] & \cdots \ar^-{d'}[r] & \on{Hom}(B_q^j, B_q^j) \ar[r] &0
	}
	$$
	and $d'$ cannot be surjective, contradiction. Hence we have proved that $B_m^j = 0$ for $j < i_0$, which completes the proof.
\end{proof}

\begin{thm}
  Let $A$ be an $n$-representation finite $k$-algebra, and let $i_0\in \Z_{\geq 0}$.
  Let $C_{n+1}\in\mathcal A_{i_0}$ be indecomposable noninjective, and let 
	$$
	C_\bullet =
	\xymatrix{
		0 \ar[r]& C_{n+1} \ar[r]^{f_{n+1}}& C_{n} \ar[r] & \cdots \ar[r]& C_1 \ar[r]^{f_1}& C_0 \ar[r] &0
	}
	$$
	be the corresponding $n$-almost split sequence. Then there are complexes 
	$A_\bullet^0\in \mathcal  C_r(\mathcal A_{i_0})$, $A_\bullet^1\in \mathcal  C_r(\mathcal A_{i_0+1})$, and a 
	radical morphism of complexes $\varphi: A_\bullet^0\to A_\bullet^1$, such that $C_\bullet \cong \on{Cone}(\varphi)$ in $\mathcal  C(\mathcal A)$.
	\label{thm:every}
\end{thm}

\begin{proof}
	By Lemma \ref{lem:claim} we can rewrite the complex $C_\bullet$ as 
	$$
	C_m = B_m^{i_0} \oplus B_m^{i_0+1}
	$$
	where $B_m^{i_0}\in\mathcal A_{i_0}$ and $B_m^{i_0+1}\in\mathcal A_{i_0+1}$ for every $m$. Moreover, 	
	$$
	f_m = 
	\begin{bmatrix}
	  b_m^{i_0} & \xi_m \\
	  \gamma_m & b_m^{i_0+1}
	\end{bmatrix}: C_m \to C_{m-1}
	$$
	has components $b_m^{i_0} : B_m^{i_0} \to B_{m-1}^{i_0}$, $\xi_m: B_m^{i_0+1}\to B_{m-1}^{i_0}$, $\gamma_m: B_m^{i_0} \to B_{m-1}^{i_0+1}$, and
	$b_m^{i_0+1}: B_m^{i_0+1} \to B_{m-1}^{i_0+1}$. Notice that by Proposition \ref{prop:hom} it follows that $\xi_m= 0$ for all $m$. 
	Define $A_m^0 = B_{m+1}^{i_0}$, $d^{A^0}_m = -b_{m+1}^{i_0}$, $A_m^1 = B_m^{i_0+1}$, $d^{A^1}_m = b_m^{i_0+1}$ and
	$\varphi_m = -\gamma_{m+1}: A_m^0 \to A_m^1$.
	Then $\varphi:A_\bullet^0 \to A_\bullet^1$ is a chain map since
	$$
	d^{A^1}_m \varphi_m = - b_m^{i_0+1} \gamma_{m+1} = \gamma_{m} b_{m+1}^{i_0} = \varphi_{m-1} d^{A^0}_m
	$$
	where the equality
	$$
	b_m^{i_0+1} \gamma_{m+1} +\gamma_{m} b_{m+1}^{i_0} =0
	$$
	comes from the fact that $C_\bullet$ is a complex.
	Moreover, $C_\bullet \cong \on{Cone}(\varphi)$ and we are done.
\end{proof}

\begin{rmk}
  In \cite[Proposition 3.23]{Iya11} Iyama constructed certain $n$-almost split sequences as mapping cones of chain maps. 
  Our Theorem \ref{thm:every} states that in the $n$-representation finite case, every $n$-almost split sequence can in fact be realised as a mapping cone.
\end{rmk}

Given that $n$-almost split sequences are determined up to isomorphism by their endpoints, it is interesting to address the issue of uniqueness of the map $\varphi$. 
Since we are not going to need it in what follows, we do not investigate this in detail. We present however a result:

\begin{prop}
  \label{prop:unique}
  Let $A$ be an $n$-representation finite algebra. Let $A_\bullet^0, B_\bullet^0\in \mathcal C(\mathcal A_{i_0})$, $A_\bullet^1,B_\bullet^1\in \mathcal C(\mathcal A_{i_0+1})$.
  Let $\varphi:A_\bullet^0\to A_\bullet^1$ and $\psi:B_\bullet^0\to B_\bullet^1$ be chain maps.
  Then the following are equivalent:
  \begin{enumerate}
    \item $\on{Cone}(\varphi)\cong \on{Cone}(\psi)$ in $\mathcal C(\mathcal A)$.
    \item There are isomorphisms of complexes $f:A^0_\bullet\to B_\bullet^0, g: A_\bullet^1\to B^1_\bullet$ such that the diagram
      $$\xymatrix{
	A^0_\bullet \ar^f[r]\ar_\varphi[d]& B_\bullet^0\ar^\psi[d]\\
	A^1_\bullet\ar^g[r] &B_\bullet^1
      }$$
      commutes in the homotopy category $\mathcal K(\mathcal A)$.
  \end{enumerate}
\end{prop}

\begin{proof}
  Let us begin by some observations. Let 
  \begin{align*}
    \alpha_m = 
    \begin{bmatrix}
      a_m & r_m \\
      q_m & b_m
    \end{bmatrix}
    : A_{m-1}^0\oplus A_m^1 \to B_{m-1}^0\oplus B_m^1
  \end{align*}
  be a morphism of modules. Notice that by Proposition \ref{prop:hom}, we have $r_m = 0$.
  Observe now that
  \begin{align*}
    &\hspace{.55cm}( \alpha_m ) \text{ defines a chain map } \alpha: \on{Cone}(\varphi)\to \on{Cone}(\psi)\\
        &\Leftrightarrow \begin{bmatrix}
      a_{m-1} & 0 \\
      q_{m-1} & b_{m-1}
    \end{bmatrix}
    \begin{bmatrix}
      -d_{m-1}^{A^0} & 0 \\
      \varphi_{m-1}& d_m^{A^1} 
    \end{bmatrix} = 
    \begin{bmatrix}
      -d_{m-1}^{B^0} & 0 \\
      \psi_{m-1}& d_m^{B^1} 
    \end{bmatrix}
\begin{bmatrix}
      a_{m} & 0 \\
      q_{m} & b_{m}
    \end{bmatrix}
    \ \text{ for all } m\\
	&\Leftrightarrow 
	\begin{cases}
	  a_{m-1}d_{m-1}^{A^0} = d_{m-1}^{B^0}a_m \ \text{ for all } m \\
	  b_{m-1} d_{m}^{A^1} = d_m^{^1}b_m \ \text{ for all } m\\
	  b_{m-1}\varphi_{m-1} = \psi_{m-1}a_{m} + d_m^{B^1} q_m + q_{m-1}d_{m-1}^{A^0} \ \text{ for all } m
	\end{cases}
 \\
& \Leftrightarrow 
\begin{cases}
  ( a_m) \text{ defines a chain map } a: A^0[-1]_\bullet\to B^0[-1]_\bullet\\
  ( b_m ) \text{ defines a chain map } b: A_\bullet^1 \to B_{\bullet}^1\\
  b_{m-1}\varphi_{m-1} = \psi_{m-1}a_{m} + d_m^{B^1} q_m + q_{m-1}d_{m-1}^{A^0} \ \text{ for all } m. 
\end{cases}
  \end{align*}

	Now let us prove $(1)\Rightarrow (2)$. Use the same notation as above, and assume that $\alpha$ is an isomorphism. That means that $\alpha_m$ is an isomorphism for every $m$. 
	Since $A_{m-1}^0\in \mathcal A_{i_0}$ and $B_m^1\in \mathcal A_{i_0+1}$, it follows that no indecomposable direct summand of $A_{m-1}^0$ can be isomorphic to a direct summand of $B_m^1$, hence $\on{Hom}(A_{m-1}^0, B_m^1)
  	= \on{rad}(A_{m-1}^0, B_m^1)$. In particular we have that $q_m$ is a radical map. Since $\alpha_m$ has an inverse, both $a_m$ and $b_m$ have inverses modulo radical morphisms. This means that there are
	$x:B_{m-1}^0\to A_{m-1}^0, y:B_m^1\to A_m^1$ such that
	\begin{align*} 
    	a_mx &- \on{id}_{B_{m-1}^0} ,\\
    	xa_m &- \on{id}_{A_{m-1}^0} , \\
    	b_my &- \on{id}_{B_m^1} , \\
    	yb_m &- \on{id}_{B_{m}^1} 
 	 \end{align*}
	 are radical morphisms. In particular $a_mx, xa_m, b_my, yb_m$ are all invertible, hence $a_m$ and $b_m$ are isomorphisms. 
	 By the above observations, $a[1]$ and $b$ are well-defined isomorphisms of complexes, and since $$\left(d_m^{B^1} q_m + q_{m-1}d_{m-1}^{A^0} \right) : A_\bullet^0 \to B_\bullet^1$$
	 is null-homotopic we obtain that the 
	 diagram 
	 $$\xymatrix{
	   	A^0_\bullet \ar^{a[1]}[r]\ar_\varphi[d]& B_\bullet^0\ar^\psi[d]\\
		A^1_\bullet\ar^b[r] &B_\bullet^1
      	 }$$
	 commutes in $\mathcal K(\mathcal A) $ as required.

       Let us now prove $(2)\Rightarrow (1)$. Since the diagram commutes in $\mathcal K(\mathcal A)$, there is a homotopy $\left( q_m: A^0_{m-1}\to B_m^1 \right)$ such that
	$$b_{m-1}\varphi_{m-1} = \psi_{m-1}a_{m} + d_m^{B^1} q_m + q_{m-1}d_{m-1}^{A^0} \ \text{ for all } m.$$ 
	By the above observations, setting for every $m$
	$$
	\alpha_m = 
	\begin{bmatrix}
	  f_{m-1} & 0\\
	  q_m & g_m
	\end{bmatrix}: A_{m-1}^0 \oplus A_m^1 \to B_{m-1}^0 \oplus B_m^1
	$$
	defines a chain map $\alpha: \on{Cone}(\varphi)\to \on{Cone}(\psi)$. 
	It remains to check that $\alpha$ is an isomorphism, which amounts to checking that $\alpha_m$ is invertible for all $m$. Since we are assuming that $f$ and $g$ are isomorphisms, we can define
	for every $m$
	$$
	\beta_m = 
	\begin{bmatrix}
	  f_{m-1}^{-1} &0 \\
	  -g^{-1}_mq_m f^{-1}_{m-1}& g_m^{-1}
	\end{bmatrix}: B_{m-1}^0 \oplus B_m^1 \to  A_{m-1}^0 \oplus A_m^1.$$
	It is then a straightforward computation to check that $\beta_m$ is the inverse of $\alpha_m$, and we are done.
 \end{proof}

\subsection{From cones to sequences}

Since we can realise any $n$-almost split sequence as $\on{Cone}(\varphi)$ for some $\varphi$, it makes sense to relate
the property of being $n$-almost split to the properties of $\varphi$. 
Let us introduce some more notation.
For a given $X \in \mathcal  A$, we define a functor $\tilde F_X: \on{Mor}_r(\mathcal  C_r(\mathcal  A))\to \on{Mor}(\mathcal  C(k\on{mod}))$ by
mapping $\varphi: A_\bullet \to B_\bullet$ to 
$$
\tilde F_X(\varphi) = \varphi\circ-: \on{Hom}(X, A_\bullet) \to F_X(B_\bullet),
$$
where $\on{Hom}(X, A_\bullet)$ denotes the complex $\cdots \to\on{Hom}(X, A_i)\to \on{Hom}(X, A_{i-1})\to\cdots$.
This is well defined because $\varphi_0 \in \on{rad}(A_0, B_0)$.

Consider the mapping cone functor $\on{Cone}: \on{Mor}_r(\mathcal  C_r(\mathcal  A))\to\mathcal  C(\mathcal  A)$. By definition, 
this factors through the inclusion $\mathcal  C_r(\mathcal  A) \to \mathcal  C(\mathcal  A)$, and we still denote by $\on{Cone}$ the corresponding functor $
\on{Cone}: \on{Mor}_r(\mathcal  C_r(\mathcal  A))\to \mathcal  C_r(\mathcal  A)$.
We also denote by $\on{Cone}$ the mapping cone functor $\on{Cone}: \on{Mor}(\mathcal  C(k\on{mod}))\to \mathcal  C(k\on{mod})$.

\begin{lemma}
	With the above notation, we have that the diagram
	$$
	\xymatrix{
	  \on{Mor}_r(\mathcal  C_r^n(\mathcal  A)) \ar[r]^-{\tilde F_X} \ar[d]_{\on{Cone}} &
	  \on{Mor}(\mathcal  C^n(k\on{mod})) \ar[d]^{\on{Cone}} \\
		\mathcal  C_r(\mathcal  A) \ar[r]^{F_X} &
		\mathcal  C(k\on{mod})
	}
	$$
	commutes for every $X\in \mathcal  A$ and for any choice of $n\in \Z_{\geq 0}$.
	\label{lem:commut}
\end{lemma}

\begin{proof}
  Pick a morphism $\varphi: A_\bullet \to B_\bullet\in \on{Mor}_r(\mathcal  C_r^n(\mathcal  A))$.
  Then 
  \begin{align*}
    \on{Cone}(\tilde F_X(\varphi))_i  &= \on{Hom}(X, A_{i-1})\oplus F_X(B_i) = \\ &=
  \begin{cases}
    \on{Hom}(X, A_{i-1})\oplus \on{Hom}(X, B_i)\ \text{ if }\ i \neq 0\\
    \on{Hom}(X, A_{-1}) \oplus \on{rad}(X, B_0)  = \on{rad}(X, B_0)\ \text{ if }\ i = 0
  \end{cases}
  \end{align*}
  and the differential $d_i: \on{Cone}(\tilde F_X(\varphi))_i  \to  \on{Cone}(\tilde F_X(\varphi))_{i-1} $ is given by 
  $$
  d_i = 
  \begin{bmatrix}
    -d^A_{i-1}\circ - & 0 \\
    -\varphi_{i-1}\circ - & d^B_i\circ -
  \end{bmatrix}.
  $$
  On the other hand, we have 
  $$
  F_X(\on{Cone}(\varphi))_i = 
  \begin{cases}
    \on{Hom}(X, A_{i-1}\oplus B_i) = \on{Hom}(X, A_{i-1})\oplus \on{Hom}(X, B_i)\ \text{ if }\ i \neq 0\\
    \on{rad}(X, A_{-1}\oplus B_0) = \on{rad}(X, B_0)\ \text{ if }\ i = 0
  \end{cases}
  $$
  and the differential $d'_i: F_X(\on{Cone}(\varphi))_i \to F_X(\on{Cone}(\varphi))_{i-1}$ is given by
  $$
  d'_i = d^{\on{Cone}(\varphi)}_i\circ - = \begin{bmatrix}
    -d^A_{i-1}\circ - & 0 \\
    -\varphi_{i-1}\circ - & d^B_i\circ-
  \end{bmatrix}.
  $$
\end{proof}

We get a useful criterion for checking whether the cone of a chain map is an $n$-almost split sequence.

\begin{lemma}[Criterion for $n$-almost splitness]
  Let $A^0_\bullet\in \mathcal C_r^{n}(\mathcal A_{i_0}), A^1_\bullet \in \mathcal C_r^{n}(\mathcal A_{i_0+1})$ for some $i_0$. Let $\varphi: A^0_\bullet\to A_\bullet^1$
  be a chain map. Then the following are equivalent:
  \begin{enumerate}
    \item $\on{Cone}(\varphi)$ is an $n$-almost split sequence.
    \item $A^0_{n}$ and $A^1_0$ are indecomposable, and $\tilde F_X(\varphi)$ is a quasi-isomorphism for every $X\in\mathcal A$.
  \end{enumerate}

	\label{lem:criterium}
\end{lemma}

\begin{proof}
  $(1)\Rightarrow (2)$. Suppose that $\on{Cone}(\varphi)$ is $n$-almost split. Then by definition $A^0_{n} = \on{Cone}(\varphi)_{n+1}$ and $A^1_0 = \on{Cone}(\varphi)_0$ 
  are indecomposable and $F_X(\on{Cone}(\varphi))$ is exact for every $X \in\mathcal A $.
  By Lemma \ref{lem:commut} we know that $F_X(\on{Cone}(\varphi)) = \on{Cone}(\tilde F_X(\varphi)) $, and  
  by Lemma \ref{lem:qiso} exactness of $\on{Cone}(\tilde F_X(\varphi))$ implies that $\tilde F_X(\varphi)$ is a quasi-isomorphism.

  $(2)\Rightarrow (1)$. If $\tilde F_X(\varphi)$ is a quasi-isomorphism for every $X\in\mathcal A$, then 
  by Lemma \ref{lem:qiso} we know that $\on{Cone}(\tilde F_X(\varphi))$ is exact, so by Lemma \ref{lem:commut} we get that $F_X(\on{Cone}(\varphi))$ is exact for every $X\in\mathcal A$.
  Then by observing that $\on{Cone}(\varphi)_{n+1} = A^0_{n}$ and $\on{Cone}(\varphi)_0 = A^1_0$ are indecomposable, 
  we can conclude that $\on{Cone}(\varphi)$ is $n$-almost split.
\end{proof}

\section{Tensor product of mapping cones}
\subsection{Construction}

All tensor products are understood to be over $k$, even when it is not specified to simplify the notation.
The tensor product bifunctor 
\begin{align*} 
  - \otimes - : \on{mod}k\times \on{mod}k\to \on{mod}k
\end{align*}
induces (for a general construction, see \cite[IV.4,5]{CE56}) a bifunctor
\begin{align*}
  - \otimes^T - : \mathcal C(\on{mod}k)\times \mathcal C(\on{mod}k)\to \mathcal C(\on{mod}k)
\end{align*}
(for clarity, we use the symbol $\otimes$ for modules and $\otimes^T$ for complexes). 
Moreover, since the tensor product defines a bifunctor
\begin{align*}
  -\otimes - : \on{mod}A\times\on{mod}B\to \on{mod}(A\otimes B)
\end{align*}
we can consider $\otimes^T$ as a bifunctor
\begin{align*}
  -\otimes^T - : \mathcal C(\on{mod}A)\times \mathcal C(\on{mod}B)\to \mathcal C(\on{mod}A\otimes B).
\end{align*}

For convenience, we give the explicit formulas: on objects, we have
\begin{align*}
  (A\otimes^T B)_m = \bigoplus_{j \in \Z} A_j\otimes B_{m-j}
\end{align*}
with differential $d$ given on an element $v\otimes w \in A_j \otimes B_{m-j}$ by
\begin{align*}
  d_m(v\otimes w) = d^A_j(v)\otimes w + (-1)^jv\otimes d^B_{m-j}(w).
\end{align*}
On morphisms, if $\varphi: A_\bullet^0 \to A_{\bullet}^1$ and $\psi: B_\bullet^0 \to B_\bullet^1$ are chain maps, then 
\begin{align*}
  (\varphi\otimes^T\psi)_m = \bigoplus_{j\in \Z} \varphi_j\otimes \psi_{m-j}: \bigoplus_{j \in \Z} A_j^0\otimes B_{m-j}^0 \to\bigoplus_{j \in \Z} A_j^1\otimes B_{m-j}^1.
\end{align*}

\begin{lemma}
  \label{lem:kun}
  Let $A, B$ be finite-dimensional $k$-algebras, let $\mathcal  A, \mathcal  B$ be subcategories of $\on{mod}A$ and $\on{mod}B$ respectively, and let $\varphi: A_\bullet^0\to A_\bullet^1$ and
  $\psi: B_\bullet^0\to B_\bullet^1$ be objects of $\on{Mor}(\mathcal  C(\mathcal  A))$ and $\on{Mor}(\mathcal  C(\mathcal  B))$ respectively. Suppose that both $\varphi$ and $\psi$ are quasi-isomorphisms.
  Then $\varphi\otimes^T\psi$ is a quasi-isomorphism.
\end{lemma}
\begin{proof}
  This follows from the K\"unneth formula over a field (see \cite[VI.3.3.1]{CE56}). That is, for complexes $A_\bullet$ and $B_\bullet$ there is for every $n$ a functorial isomorphism
  $$
  H_n(A_\bullet\otimes^T B_\bullet) \cong \bigoplus_{i+j = n}H_{i}(A_\bullet) \otimes H_j(B_\bullet).
  $$
  In our case, this gives for every $n$ an isomorphism
  $$
  H_n(\varphi \otimes^T\psi) \cong \left( H_i(\varphi)\otimes H_{j}\left( \psi \right) \right)_{i+j= n}.
  $$
  Since $\varphi$ and $\psi$ are quasi-isomorphisms, the right-hand side is an isomorphism, hence $\varphi\otimes^T\psi$ is a quasi-isomorphism.
\end{proof}

\subsection{Preparation}

We now focus on the tensor product of homogeneous algebras. In this case the tensor product behaves well (recall that we are assuming $k$ to be perfect).
More precisely, we have the following classical result:

\begin{prop}
  Let $A, B$ be finite-dimensional $k$-algebras. Then 
  \label{prop:gldim}
  $$
  \on{gl.dim}(A\otimes_k B) = \on{gl.dim}(A) + \on{gl.dim}(B).
  $$
\end{prop}

\begin{proof}
  Using a result by Auslander (\cite[Theorem 16]{Aus55}),
  we can assume that $A$ and $B$ are semisimple. Then the claim is a special case of \cite[Corollary 5.7]{Kre79}.
\end{proof}

In our setting, perfectness of the ground field and homogeneity are enough to guarantee that higher representation finiteness is preserved by tensor products:

\begin{thm}
  \label{thm:martin}
  Let $A$ be an $n$-representation finite $k$-algebra, and let $B$ be an $m$-representation finite $k$-algebra. 
  If $A$ and $B$ are $l$-homogeneous, then the algebra $A\otimes_k B$ is
  $(n+m)$-representation finite, $l$-homogeneous. Moreover, an $(n+m)$-cluster tilting module for $A\otimes_k B$ is given by
  $$
  M_{A\otimes B} = \bigoplus_{i= 0}^{l-1} \tau_n^{-i} A\otimes \tau_m^{-i}B.
  $$
\end{thm}

\begin{proof}
  See \cite[1.5]{HI11}.
\end{proof}

\begin{prop}
  Let $A$ and $B$ be two finite-dimensional $k$-algebras. Let $M, N \in\on{mod}A$ and $M', N'\in \on{mod}B$.
  Then the canonical map
  \begin{align*}\label{eq:hom}
    \on{Hom}_A(M, N) \otimes_k  \on{Hom}_B(M', N') \to \on{Hom}_{A\otimes_k B}(M\otimes_k M', N\otimes_k N') 
  \end{align*}
    given by $f\otimes g\mapsto f\otimes g$
  is an isomorphism of $k$-vector spaces.
  \label{prop:iso}
\end{prop}

\begin{proof}
  See Proposition XI.1.2.3 and Theorem XI.3.1 in \cite{CE56}.
\end{proof}

We will use the above identification quite freely from now on.
We need two more lemmas:
\begin{lemma}
  \label{lem:rad1}
  Let $R$ and $S$ be finite-dimensional $k$-algebras. Then we have
      $$
      \on{rad}(R) \otimes_k S + R\otimes_k \on{rad}(S) = \on{rad}(R\otimes_k S)
      $$
      as ideals of $R\otimes_k S$.
 \end{lemma}
\begin{proof}
  This is \cite[Corollary 5.8]{Kre79}, combined with the observation that for finite-dimensional algebras the Baer radical and the Jacobson radical coincide (see \cite[Proposition 10.27]{Lam01}).
      \end{proof}

 \begin{lemma}
   Let $A$ and $B$ be two finite-dimensional $k$-algebras. Let $M, N\in \on{mod}A$ and $M', N' \in\on{mod}B$. 
   Then we have
   $$
   \on{rad}(M, N)\otimes\on{Hom}(M',N') + \on{Hom}(M, N)\otimes \on{rad}(M', N') = \on{rad}(M\otimes M', N\otimes N')
   $$
   as subspaces of $\on{Hom}(M\otimes M', N\otimes N')$.
   Moreover, there is an exact sequence
		    $$
	    \xymatrix{
			0 \ar[r] & \on{rad}(M)\otimes \on{rad}(M') \ar^-*{\left[\begin{smallmatrix}
				\alpha \\-\alpha
			\end{smallmatrix}\right]}[r] 
				& *+{\begin{smallmatrix}\on{rad}(M) \otimes \on{End}(M')\\ \bigoplus \\ \on{End}(M)\otimes \on{rad}(M') \end{smallmatrix}}\ar^-*{\left[\begin{smallmatrix}
				\alpha &\alpha
			    \end{smallmatrix}\right]}[r]
				& \on{rad}(M\otimes M') \ar[r] &0
			      }
			  $$
	where 
$$
\alpha: f\otimes g \mapsto f\otimes g.
$$
   \label{lem:rad}
 \end{lemma}

\begin{proof}
  Let $R = \on{End}_A(M\oplus N)$ and $S= \on{End}_B(M'\oplus N')$.
  By Proposition \ref{prop:iso} we have 
  $$
  R\otimes S \cong \on{End}_{A\otimes B}((M\oplus N) \otimes (M'\oplus N')).
  $$
  Let $p, q\in R$ be the projections onto $M, N$ respectively, and let $p', q'\in S$ be the projections onto $M', N'$ respectively. 
  Then we have
  \begin{align*}
    (q\otimes q') (\on{rad}(R\otimes S))(p\otimes p') = \on{rad}(M\otimes M', N\otimes N').
  \end{align*}
  By Lemma \ref{lem:rad1}, 
  \begin{align*}
    \on{rad}(R\otimes S) =  \on{rad}(R) \otimes S + R\otimes \on{rad}(S) 
  \end{align*}
  so that
  \begin{align*}
    \on{rad}(M\otimes M', N\otimes N') &= (q\otimes q') (\on{rad}(R) \otimes S + R\otimes \on{rad}(S))(p\otimes p') = \\
    &=\on{rad}(M, N)\otimes\on{Hom}(M',N') + \on{Hom}(M, N)\otimes \on{rad}(M', N'),
  \end{align*}
  which proves the first claim.
  Moreover, in the case $M= N, M'=N'$ we easily get the exact sequence 
  by looking at the kernel of the map
  $$
 \xymatrix{
	*{\left[\begin{matrix}
				\alpha & \alpha
			\end{matrix}\right]}& 
			*---{:}&
			*+{\begin{smallmatrix}\on{rad}(M) \otimes \on{End}(M')\\ \bigoplus \\ \on{End}(M)\otimes \on{rad}(M') \end{smallmatrix}}
			\ar[r] &
			\on{rad}(M\otimes M').
	}
$$
\end{proof}

\subsection{Proof of main result} 

We are ready to prove Theorem $\ref{thm:main}$:

\begin{proof}[Proof of Theorem \ref{thm:main}]
  We fix $\varphi: A_\bullet^0\to A_\bullet^1$ and $\psi: B_\bullet^0\to B_\bullet^1$.
  By definition $C_\bullet= \on{Cone}(\varphi\otimes^T\psi)$ is a complex bounded between 0 and $n+m+1$, and it is exact by Lemma \ref{lem:qiso} and Lemma \ref{lem:kun}. 
  It follows from Lemma \ref{lem:rad} that $$
  (\varphi\otimes^T \psi)_i \in \on{rad}((A_{\bullet}^0\otimes^T B_\bullet^0)_i, (A_\bullet^1\otimes^T B_\bullet^1)_i)
  $$
  for every $i$, and so $C_\bullet \in \mathcal  C_r(\mathcal {A\otimes B})$.
	Fix an indecomposable $M\otimes N\in\mathcal {A\otimes B}$.
	We can consider the maps 
  	$$
  	\tilde F_M(\varphi)\otimes^T \tilde F_N(\psi) : \on{Hom}(M, A^0_\bullet)\otimes^T \on{Hom}(N, B^0_\bullet)
	\to F_M(A^1_\bullet) \otimes^T F_N(B^1_\bullet)
 	$$
  	and
  	$$
	\tilde F_{M\otimes N}(\varphi\otimes^T \psi) : \on{Hom}(M\otimes N, A^0_\bullet\otimes^T B^0_\bullet) \to F_{M\otimes N}( A^1_\bullet\otimes^T B^1_\bullet).
  	$$
	By Lemma \ref{lem:rad}, the map $$\iota: \on{Hom}(M, A_{\bullet}^1) \otimes^T \on{Hom}(N, B_\bullet^1) \to \on{Hom}(M\otimes N, A^1_\bullet\otimes^T B^1_\bullet),
	\ \ f\otimes g\mapsto f\otimes g$$
  	induces a monomorphism
  	$$
	\iota': F_M(A^1_\bullet) \otimes^T F_N(B_\bullet^1) \to F_{M\otimes N}(A_\bullet^1\otimes^T B_\bullet^1)
  	$$
  	so there is a commutative diagram
	$$
	\xymatrix{
	  \on{Hom}(M, A^0_\bullet)\otimes^T \on{Hom}(N, B^0_\bullet) \ar[r]^-\iota\ar[d]_*{ \tilde F_M(\varphi)\otimes^T \tilde F_N(\psi) } 
	  & \on{Hom}(M\otimes N, A^0_\bullet\otimes^T B^0_\bullet) \ar[d]^*{\tilde F_{M\otimes N}(\varphi\otimes^T \psi) }\\
	  F_M(A^1_\bullet) \otimes^T F_N(B^1_\bullet)\ar[r]_-{\iota'}
	  &  F_{M\otimes N}(A_\bullet^1\otimes^T B_\bullet^1).
	}
	$$
	By Proposition \ref{prop:iso}, the map $\iota$ is an isomorphism. Moreover, since $\on{Cone}(\varphi)$ is $n$-almost split, it follows by Lemma \ref{lem:criterium} that $\tilde F_M(\varphi)$ is a quasi-isomorphism, and
	similarly $\tilde F_N(\psi)$ is a quasi-isomorphism because $\on{Cone}(\psi)$ is $m$-almost split.
	Then by Lemma \ref{lem:kun} it follows that $ \tilde F_M(\varphi)\otimes^T \tilde F_N(\psi)$ is a quasi-isomorphism.
	Again by Lemma \ref{lem:criterium}, the claim that $C_\bullet$ is $(m+n)$-almost split will follow if we prove that $\tilde F_{M\otimes N}(\varphi \otimes^T \psi)$ is a quasi-isomorphism (since $M\otimes N$ is an arbitrary
	indecomposable).
  	By the above observations, it is enough to show that $\iota'$ is a quasi-isomorphism. This is in turn equivalent to $\on{coker} \iota'$ being exact, which is what we prove.
  	We claim that we have
	\begin{equation}\label{eq:coker}
		\on{coker} \iota' = F_M(A_\bullet^1)\otimes_k S(N,B_0^1) \oplus S(M,A_0^1)\otimes_k F_N(B_\bullet^1).
	\end{equation}
	Assume that this claim holds, and let us prove the theorem.
	Notice that $S(N, B_0^1) = 0$ unless $N \cong B_0^1$ since $B_{0}^1$ and $N$ are indecomposable.
	Suppose that $N \cong B_0^1$. Then in particular $N \in \on{add}\tau_m^{-(i+1)}B$ and so $M\in \on{add}\tau_n^{-(i+1)}A$ since $M\otimes N\in \mathcal {A\otimes B}$ (see Theorem \ref{thm:martin}). Then by Proposition \ref{prop:hom} we get 
	$$\on{Hom}(M,A_\bullet^0) = 0.$$
	In this case $F_M(A_\bullet^1) \cong \on{Cone}(\tilde F_M(\varphi))$ which by Lemma \ref{lem:criterium} is exact if and only if $\on{Cone}(\varphi)$ is $n$-almost split, which we are assuming. 
	Tensoring over $k$ is exact, so it follows that the first summand in $(\ref{eq:coker})$ is exact. By symmetry, the second summand is exact as well and we are done.
		
	It remains to prove the equality (\ref{eq:coker}).
	Call $D_\bullet = F_M(A_\bullet^1)\otimes^T F_N(B_\bullet^1)$. We have that 
	$$D_p = \bigoplus_{i+j= p}F_M(A_\bullet^1)_i\otimes F_N(B_\bullet^1)_j$$
	and we are interested in computing the cokernels of the maps
	$$
	\iota'_p: D_p \to F_{M\otimes N}\left(A_\bullet^1\otimes^T B_\bullet^1\right)_p.
	$$
	We proceed by first considering the case $p \neq 0$. Then the codomain of $\iota'_p$ is 
	$$
	\on{Hom}\left(M\otimes N, \bigoplus_{i+j=p}A_i^1\otimes B_j^1\right) \cong \bigoplus_{i+j=p}\on{Hom}(M, A_i^1)\otimes \on{Hom}(N, B_j^1)
	$$
	and $\iota'_p$ is just the canonical diagonal immersion with components 
	$$
	\iota'_{ij}: F_M(A_\bullet^1)_i\otimes F_N(B_\bullet^1)_j \to \on{Hom}(M,A_i^1)\otimes \on{Hom}(N,B_j^1)
	$$
	given by $f\otimes g \mapsto f\otimes g$.
	In particular, $\iota'_{ij}$ is the identity unless either $i=0$ and $M\cong A_0^1$ or $j=0$ and $N\cong B_0^1$.
	It follows that $$
	\on{coker}\iota'_p = \bigoplus_{i+j=p}\on{coker}\iota'_{ij} = \on{coker}\iota'_{0p} \oplus \on{coker}\iota'_{p0}.
	$$
	Let us then suppose $N\cong B_0^1$, and focus on terms of the form $\on{coker}\iota'_{p0}$, where
	$$\iota'_{p0}: \on{Hom}(M, A_p^1)\otimes \on{rad}(B_0^1) \to \on{Hom}(M\otimes B_0^1, A_p^1\otimes B_0^1).$$
	We know by Proposition \ref{prop:iso} that the right-hand side is canonically isomorphic to $\on{Hom}(M,A_p^1)\otimes \on{End}(B_0^1)$, so from 
	the exact sequence
	$$
	\xymatrix{
		0 \ar[r] &
		\on{rad}(B_0^1) \ar[r] &
		\on{End}(B_0^1)\ar[r] &
		S(B_0^1) \ar[r]
		&0
	}
	$$
	we conclude that $\on{coker}\iota'_{p0} = \on{Hom}(M,A_p^1)\otimes S(B_0^1)$.
	By symmetry we conclude that if $p \neq 0$ then 
	$$
	\on{coker}\iota'_p = \on{Hom}(M,A_p^1)\otimes S(N, B_0^1) \oplus S(M, A_0^1)\otimes \on{Hom}(N,B_p^1).
	$$
	Let us analyse the case $p = 0$. Under the identification given by Proposition \ref{prop:iso}, the map
	$$
	\iota'_0 : \on{rad}(M, A_0^1) \otimes \on{rad}(N, B_0^1)\to \on{rad}(M\otimes N, A_0^1\otimes B_0^1)
	$$
	is the identity if $M\not\cong A_0^1$ and $N\not \cong B_0^1$, and the inclusion otherwise.
	If $M \not \cong A_0^1$ and $N\cong B_0^1$, then we are in the same situation as in the previous case, and
	$$
	\on{coker}\iota'_0 = \on{Hom}(M, A_0^1)\otimes S(B_0^1)
	$$
	and similarly for the symmetric case. If both $M\cong A_0^1$ and $N\cong B_0^1$, then we claim that 
	$$\on{coker}\iota'_0 = \on{rad}(A_0^1)\otimes S(B_0^1)\oplus S(A_0^1)\otimes \on{rad}(B_0^1).$$
	Indeed (for simplicity, write $E = A_0^1$ and $F= B_0^1$),
			in the commutative diagram	
			$$
			\xymatrix{
				&0 \ar[d] & 0 \ar[d] & 0 \ar[d] & \\
				0 \ar[r] & \on{rad}(E)\otimes \on{rad}(F) \ar^-*{\left[\begin{smallmatrix}
				\alpha\\-\alpha
			\end{smallmatrix}\right]}[r] \ar@{=}[d]
			&*+{\begin{smallmatrix}\on{rad}(E)\otimes \on{rad}(F)\\ \bigoplus \\\on{rad}(E)\otimes \on{rad}(F)\end{smallmatrix}}
				\ar^-*{\left[\begin{smallmatrix}
				\alpha&  \alpha
			\end{smallmatrix}\right]}[r] \ar^*+{\left[\begin{smallmatrix}
			\alpha & 0\\0 & \alpha
		\end{smallmatrix}\right]}[d] 
		& \on{rad}(E)\otimes \on{rad}(F) \ar[r] \ar[d]^{\iota'_0} &0\\
				0 \ar[r] & \on{rad}(E)\otimes \on{rad}(F) \ar^-*{\left[\begin{smallmatrix}
				\alpha \\-\alpha
			\end{smallmatrix}\right]}[r] \ar[d]
				& *+{\begin{smallmatrix}\on{rad}(E) \otimes \on{End}(F)\\ \bigoplus \\ \on{End}(E)\otimes \on{rad}(F) \end{smallmatrix}}\ar^-*{\left[\begin{smallmatrix}
				\alpha &\alpha
			    \end{smallmatrix}\right]}[r] \ar[d]
				& \on{rad}(E\otimes F) \ar[r]\ar[d] &0\\
				& 0 & *+{\begin{smallmatrix}\on{rad}(E) \otimes S(F)\\ \bigoplus\\ S(E) \otimes \on{rad}(F)\end{smallmatrix}} \ar[d] 
					& \on{coker}\iota'_0 \ar[d] &\\
				&&0&0&
			}
			$$
			the first row is exact, as well as all the columns ($\alpha$ denotes the canonical map $f\otimes g\mapsto f\otimes g$). The second row is exact by Lemma \ref{lem:rad}. Hence we get an isomorphism
		$$\on{rad}(E) \otimes S(F) \oplus S(E) \otimes \on{rad}(F)\cong \on{coker}\iota'_0$$ by the $3\times 3$ lemma.
		We have shown that 
		$$
		\on{coker}\iota'_p = F_M(A_p^1)\otimes S(N,B_0^1) \oplus S(M,A_0^1)\otimes F_N(B_p^1)
		$$
		for every value of $p = 0, \ldots, m+n$.	

		It remains to show that the differentials $\on{coker}\iota'_{p+1}\to \on{coker}\iota'_p$ 
		are diagonal, to conclude that the direct-sum decomposition of the objects is actually a direct-sum 
		decomposition into the two complexes appearing in equation (\ref{eq:coker}).
		The only degree where this poses problems is $p = 0$ in the case $M \cong E = A_0^1,\ N\cong F=B_0^1$. For this, consider the following diagram: 
		\vspace{.3cm}
		$$
		\xymatrix{
			*{\begin{smallmatrix}
				\on{Hom}(E, A_1^1)\otimes \on{rad}(F)\\
				\bigoplus\\
				\on{rad}(E)\otimes\on{Hom}(F, B_1^1)
		\end{smallmatrix}}
		\ar^{\beta}[r]\ar[d]_{\iota'_1}
		&
		\on{rad}(E)\otimes\on{rad}(F)
		\ar^{\iota'_0}[d]\\
		*{\begin{smallmatrix}
			\on{Hom}(E,A_1^1)\otimes \on{End}(F)\\
			\bigoplus\\
			\on{End}(E)\otimes \on{Hom}(F, B_1^1)
		\end{smallmatrix}}
		\ar^{\beta}[r]\ar[d]&
		\on{rad}(E\otimes F)
		\ar[d]\\
		*{\begin{smallmatrix}
			\on{Hom}(E,A_1^1)\otimes S(F)\\
			\bigoplus\\
			S(E)\otimes\on{Hom}(F,B_1^1)
		\end{smallmatrix}}
		\ar[r]
		& \on{coker}\iota'_0
		}
		\vspace{.3cm}
		$$
	      where the horizontal maps are induced by $$\beta = \begin{bmatrix}
	      	(d^A_1\circ -)\otimes \on{id}, & \on{id}\otimes (d^B_1\circ -)
	      \end{bmatrix},$$ which is the last
		map appearing in 
		the sequence $F_E(A_\bullet^1)\otimes^T F_F(B_\bullet^1)$.
		The map $\beta$ factors as
		$$
		\beta=	
			\begin{bmatrix}
				\alpha&\alpha
			\end{bmatrix}
			\begin{bmatrix}
				(d^A_1\circ -)\otimes \on{id}&0\\
				0	& \on{id}\otimes (d^B_1\circ -)
			\end{bmatrix}
				$$
		hence the diagram above can be completed to a diagram
		\vspace{.3cm}
		$$
		\xymatrix{
		*{\begin{smallmatrix}
				\on{Hom}(E,A_1^1)\otimes \on{rad}(F)\\
				\bigoplus\\
				\on{rad}(E)\otimes\on{Hom}(F, B_1^1)
		\end{smallmatrix}}
		\ar[r]\ar[d]_{\iota'_1}
		&
		*{\begin{smallmatrix}\on{rad}(E)\otimes \on{rad}(F)\\
				\bigoplus\\
				\on{rad}(E)\otimes \on{rad}(F)\end{smallmatrix}}
			      \ar^-*{\left[ \begin{smallmatrix}
				\alpha & \alpha
			    \end{smallmatrix}\right]}[r]\ar[d]
		&
		\on{rad}(E)\otimes\on{rad}(F)
		\ar^{\iota'_0}[d]\\
		*{\begin{smallmatrix}
			\on{Hom}(E,A_1^1)\otimes \on{End}(F)\\
			\bigoplus\\
			\on{End}(E)\otimes \on{Hom}(F,B_1^1)
		\end{smallmatrix}}
		\ar[r]\ar[d]
		&	
		*{\begin{smallmatrix}
			\on{rad}(E)\otimes \on{End}(F)\\
			\bigoplus\\
			\on{End}(E)\otimes \on{rad}(F)
		\end{smallmatrix}}
		\ar^-*{\left[ \begin{smallmatrix}
				\alpha & \alpha
			    \end{smallmatrix}\right]}[r]\ar[d]
		&
		\on{rad}(E\otimes F)
		\ar[d]\\
		*{\begin{smallmatrix}
			\on{Hom}(E,A_1^1)\otimes S(F)\\
			\bigoplus\\
			S(E)\otimes\on{Hom}(F,B_1^1)
		\end{smallmatrix}}
		\ar[r]&
		*{\begin{smallmatrix}
			\on{rad}(E)\otimes S(F)\\
			\bigoplus\\
			S(E)\otimes \on{rad}(F)
		\end{smallmatrix}}
		\ar^-*{\cong}[r]&	\on{coker}\iota'_0\\
		},
		\vspace{.3cm}
		$$
		where the horizontal maps on the left-hand side are diagonal.
		Hence the induced map 
		$$
		\xymatrix{
			*{\begin{smallmatrix}
			\on{Hom}(E,A_1^1)\otimes S(F)\\
			\bigoplus\\
			S(E)\otimes\on{Hom}(F,B_1^1)
		\end{smallmatrix}}
		\ar[rr]&&
		\on{coker}\iota'_0
		}
		$$
		factors through the diagonal map 
		$$
		\xymatrix{
			*{\left[\begin{smallmatrix}
				(d^A_1\circ -)\otimes \on{id}_{S(F)}&0\\
				0	& \on{id}_{S(E)}\otimes (d^B_1\circ -)
			\end{smallmatrix}\right]}
			&*---{:}
			&
			*{\begin{smallmatrix}
			\on{Hom}(E,A_1^1)\otimes S(F)\\
			\bigoplus\\
			S(E)\otimes\on{Hom}(F,B_1^1)
			\end{smallmatrix}}
			\ar[r]&
			*{\begin{smallmatrix}
			\on{rad}(E)\otimes S(F)\\
			\bigoplus\\
			S(E)\otimes \on{rad}(F)
			\end{smallmatrix}}
		}
		$$
		and we are done.

\end{proof}

\section{Examples}
As an example, consider the quiver
$$
\xymatrix{
  1\ar[r]&2\ar[r]& 3&4\ar[l]&5\ar[l]
}
$$
and the corresponding path algebra $A = kQ$.
Thus $A$ is 3-homogeneous, (1\nolinebreak-) representation finite (see \cite{ASS06}, \cite{HI11}). We want to consider the algebra $B = A\otimes A$, which is then 3-homogeneous, 2-representation finite.
There are 15 nonisomorphic indecomposables in $\on{mod}A$, which have the following dimension vectors:
\begin{align*}
  \begin{matrix}
    P_1 : (11100) & M_1 : (01111) & I_1: (10000)\\
    P_2 : (01100) & M_2 : (01000) & I_2: (11000)\\
    P_3 : (00100) & M_3 : (01110) & I_3: (11111)\\
    P_4 : (00110) & M_4 : (00010) & I_4: (00011)\\
    P_5 : (00111) & M_5 : (11110) & I_5: (00001).\\
  \end{matrix}
\end{align*}
The Auslander-Reiten quiver of $A$ is the following:
$$
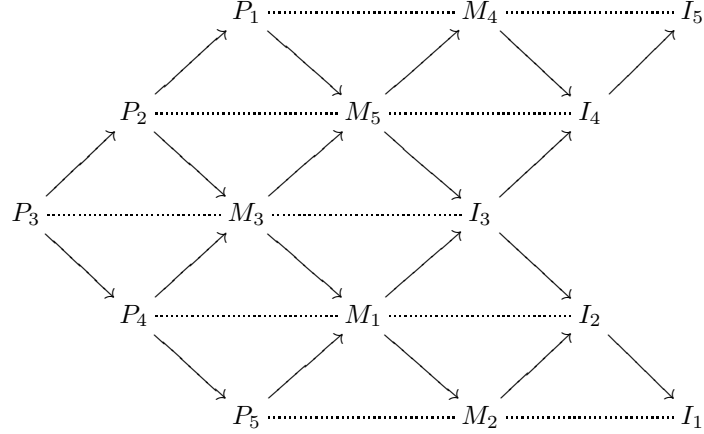
\xymatrix{
  & & P_1 \ar@{.}[rr]\ar[dr] & & M_4 \ar@{.}[rr]\ar[dr]&&I_5\\
  &P_2\ar[ru]\ar@{.}[rr]\ar[dr]& & M_5\ar[ru]\ar[dr]\ar@{.}[rr]& & I_4\ar[ru]&\\
  P_3\ar[ru]\ar@{.}[rr]\ar[dr]& & M_3\ar[ru]\ar[dr]\ar@{.}[rr]& & I_3\ar[ru]\ar[dr]&&\\
  &P_4\ar[ru]\ar@{.}[rr]\ar[dr]& & M_1\ar[ru]\ar[dr]\ar@{.}[rr]& & I_2\ar[rd]&\\
  &&P_5\ar[ru]\ar@{.}[rr]& & M_2\ar[ru]\ar@{.}[rr]& & I_1
}
$$
where the dotted lines represent $\tau^-_{1}$.

Inside $\on{mod}B$ we have the 2-cluster tilting subcategory $\mathcal  C = \on{add}M$, where $$M = \bigoplus_{1\leq i,j\leq  5}P_i\otimes P_j\oplus \bigoplus_{1\leq i,j\leq 5}M_i\otimes M_j\oplus \bigoplus_{1\leq i,j\leq  5} I_i\otimes I_j.$$
Let us consider for instance the (1-)almost split sequences
$$
C_\bullet = 
\xymatrix{
  0\ar[r] 
  & P_2\ar^-*{\left[\begin{smallmatrix}
  a\\b
\end{smallmatrix}\right]}[r]
  & *{P_1\oplus M_3}\ar^-*{\left[\begin{smallmatrix}
  c&d
\end{smallmatrix}\right]}[r]
  &M_5\ar[r]
  &0
}
$$
and
$$
D_\bullet = 
\xymatrix{
  0\ar[r] 
  & P_5\ar^e[r]
  & M_1\ar^f[r]
  &M_2\ar[r]
  &0
}
$$
in $\on{mod}A$.
Notice that both these sequences start in slice 0.
The sequence $C_\bullet$ is isomorphic to the cone of 
$$
\xymatrix{
  \cdots \ar[r]& 0\ar[r]\ar[d] & P_2 \ar^{-a}[r]\ar_{-b}[d]& P_1\ar[r]\ar_{-c}[d]& 0 \ar[r]\ar[d]& \cdots\\
  \cdots \ar[r]& 0\ar[r] & M_3\ar^d[r]& M_5\ar[r] &0 \ar[r]&\cdots
}
$$
and $D_\bullet$ is isomorphic to the cone of 
$$
\xymatrix{
  \cdots\ar[r]&0\ar[r]\ar[d]& P_5 \ar[r]\ar[d]_{-e}&0\ar[r]\ar[d]&0\ar[r]\ar[d]&\cdots\\
  \cdots\ar[r]&0\ar[r]& M_1\ar[r]^f&M_2\ar[r]&0\ar[r]&\cdots
}
$$
where these diagrams should be seen as morphisms $\varphi, \psi$ of chain complexes.
Then we can construct the morphism $\varphi\otimes^T\psi$:
$$
\xymatrix{
  \cdots\ar[r]&0\ar[r]\ar[d]& P_2\otimes P_5 \ar[r]^{-a\otimes 1}\ar[d]_{b\otimes e}& P_1\otimes P_5\ar[rr]\ar[d]^*{\left[\begin{smallmatrix}
  0\\c\otimes e
\end{smallmatrix}\right]}&& 0\ar[r]\ar[d]&0\ar[r]\ar[d]&\cdots\\
\cdots\ar[r]&0\ar[r]& M_3\otimes M_1\ar[r]^*{\left[\begin{smallmatrix}
  -1\otimes f\\ d\otimes 1
\end{smallmatrix}\right]} &
  *{\begin{smallmatrix}
  M_3\otimes M_2\\\oplus\\M_5\otimes M_1
\end{smallmatrix}}\ar[rr]^-*{\left[ \begin{smallmatrix}
  d\otimes 1&1\otimes f
\end{smallmatrix}\right]}&&M_5\otimes M_2\ar[r]&0\ar[r]&\cdots
}
$$
The cone $E_\bullet = \on{Cone}(\varphi\otimes^T\psi)$ is then the sequence
$$
\xymatrix{
  0\ar[r]& P_2\otimes P_5\ar[r]^*{\left[ \begin{smallmatrix}
    a\otimes 1 \\- b\otimes e
\end{smallmatrix}\right]} 
  & *{\begin{smallmatrix}
  P_1\otimes P_5 \\
  \oplus\\
  M_3\otimes M_1
\end{smallmatrix}}\ar[rr]^*{\left[ \begin{smallmatrix}
  0 &-1\otimes f\\
  -c\otimes e&d\otimes 1
\end{smallmatrix}\right]}
&& *{\begin{smallmatrix}
  M_3\otimes M_2\\\oplus\\M_5\otimes M_1
\end{smallmatrix}}\ar[rr]^*{\left[\begin{smallmatrix}
  d\otimes 1&1\otimes f
\end{smallmatrix}\right]}
&& M_5\otimes M_2\ar[r]
& 0
}
$$
which is 2-almost split in $\mathcal  C$ by Theorem \ref{thm:main}.

Now we can go further, and consider the algebra $B\otimes A$, which is then 3-homogeneous, 3-representation finite. 
Let us write for simplicity $P_{abc} = P_a\otimes P_b\otimes P_c$ and $M_{abc} = M_a\otimes M_b\otimes M_c$. We look at the 3-almost split sequence starting in $P_{254}$, which is 
obtained from $E_\bullet$ together with the sequence
$$
\xymatrix{
  0\ar[r]& P_4\ar[r]& P_5\oplus M_3\ar[r]& M_1\ar[r]&0
}
$$
in $\on{mod}A$.
By applying the formula we get the sequence
$$
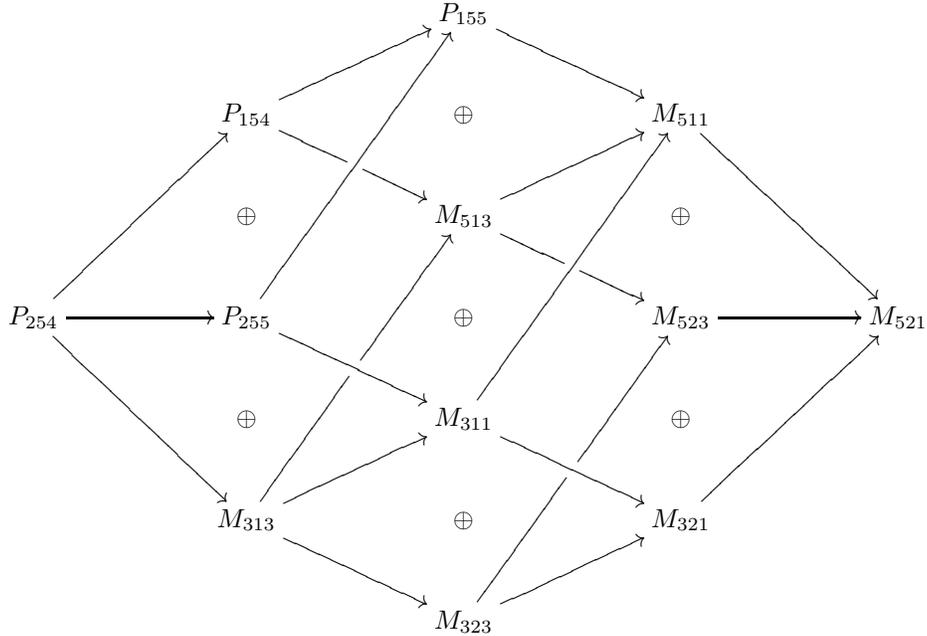
\xymatrix{
  &&&&P_{155}\ar[rrd]&&&&\\
  &&P_{154}\ar[rru]\ar[rrd]|\hole&&\oplus&&M_{511}\ar[rrdd]&&\\
  &&\oplus&&M_{513}\ar[rru]\ar[rrd]|\hole&&\oplus&&\\
  P_{254}\ar[rr]\ar[rruu]\ar[rrdd]&&P_{255}\ar[rruuu]\ar[rrd]&&\oplus&&M_{523}\ar[rr]&&M_{521}\\
  &&\oplus&&M_{311}\ar[rruuu]\ar[rrd]&&\oplus&&\\
  &&M_{313}\ar[rruuu]|\hole\ar[drr]\ar[rru]&&\oplus&&M_{321}\ar[rruu]&&\\
  &&&&M_{323}\ar[rru]\ar[rruuu]|\hole&&&&
}
$$
where each arrow is the natural morphism up to sign.
\vspace{.5cm}
\paragraph{\textbf{Acknowledgement.}} The author is thankful to his advisor Martin Herschend for the constant supervision and helpful discussions, as well as for useful ideas about some proofs.
The author would also like to thank an anonymous referee for the helpful suggestions.
The author was funded by Uppsala University.

\bibliography{Bibliography}{}
\bibliographystyle{alpha}

\end{document}